\documentclass[a4paper,11pt,twoside]{article}
\usepackage{amsmath,amsfonts,amssymb}

\usepackage{amsbsy}
\usepackage{amssymb}
\usepackage{amsfonts}
\usepackage{xcolor}
\usepackage{amsmath}
\usepackage{amsopn}
\usepackage{amsthm}
\usepackage[english]{babel}
\usepackage[applemac]{inputenc}
\usepackage[T1]{fontenc}
\usepackage{graphicx}
\usepackage{soul}
\usepackage{todonotes}
\usepackage[normalem]{ulem}
\newtheorem{teo}{Theorem}
\newtheorem{lemma}{Lemma}
\newtheorem{prop}{Proposition}
\newtheorem{cor}{Corollary}

\theoremstyle{definition}
\newtheorem{remark}{Remark}

\theoremstyle{definition}

\theoremstyle{definition}
\newtheorem{example}{Example}

\theoremstyle{remark}

\DeclareMathOperator{\re}{\mathbb{R}}

\DeclareMathOperator{\er}{\mathbf{E}}
\DeclareMathOperator{\pr}{\mathbf{P}}

\DeclareMathOperator{\erv}{\mathbf{e}}

\newcommand{\erva}{\mathbf{e}_\alpha}

\newcommand{\ervk}{\mathbf{e}_k}

\newenvironment{dedication}
        {\vspace{6ex}\begin{quotation}\begin{center}\begin{em}}
        {\par\end{em}\end{center}\end{quotation}}

\title{On exponential functionals, harmonic potential measures and undershoots of subordinators}

\date{}

\author{Larbi Alili\thanks{Department of Statistics, The University of Warwick. E-mail: L.Alili@warwick.ac.uk},
\and Wissem Jedidi\thanks{Department of Statistics \& OR, King Saud University. E-mail: wissem\_jedidi@yahoo.fr} \thanks{Universit\'e de
Tunis El Manar, Facult\'e des Sciences de Tunis, D\'epartement de Math\'ematiques,
Laboratoire d'Analyse Math\'ematiques et Applications LR11ES11.  2092 - El Manar I, Tunis,
Tunisia.}, \and V{\'\i }ctor Rivero\thanks{Centro de Investigaci\'on
en Matem\'aticas (CIMAT A.C.) E-mail: rivero@cimat.mx}}

\begin{document}
\maketitle
\begin{abstract} We establish a link  between the distribution of an exponential functional, $I,$ and the undershoots of a subordinator, which is given in terms of the associated harmonic potential measure. This allows us to give a necessary and sufficient condition in terms of the L\'evy measure for the exponential functional to be multiplicative infinitely divisible. We then provide a formula for the moment generating functions of $\log I$ and $\log R$ where  $R$ is the so-called remainder random variable associated to $I$. We provide a realization of the remainder random variable $R$ as an infinite product involving independent last position random variables of the subordinator. Some properties of harmonic measures are obtained and some examples are provided.
\end{abstract}
\begin{dedication}
\hspace{1cm}
\vspace*{0cm}{Dedicated to the memory of Professor Marc Yor.}
\end{dedication}

\section{Introduction and main results}
Let $\xi$ be a (possibly killed) subordinator; that is a $[0,\infty]$-valued L\'evy process with non-decreasing paths, absorbing state $\Delta:=\{+\infty\}$ and lifetime $\zeta:=\inf\{ s>0:\ \xi_s=\infty\}$.  We denote by $\phi:\re^{+}\to\re^{+}$ its Laplace exponent, i.e.
$
\phi(\lambda):=-\log\er [e^{-\lambda\xi_{1}}]$,  $\lambda\geq 0$. By the L\'evy-Khintchine formula, we know that
\begin{equation}\label{LK}
\phi(\lambda)=q+\lambda a+\int_{(0,\infty)}(1-e^{-\lambda
x})\Pi(dx),\quad \lambda\geq 0,
 \end{equation}
 where $a\geq 0$ is the drift, $q\geq 0$ is the killing term, and $\Pi$ is a measure on $(0,\infty)$ such that ${\small \int_{(0,\infty)}(x\wedge1)\,\Pi(dx)<\infty}$.  The distribution of $\xi$ is characterized by the characteristic triplet $(q, a, \Pi)$.

 Our main purpose, in this work, is to make some contributions to the description of the distribution of the so-called exponential functional of $\xi$ and its associated remainder random variable, in the terminology proposed by Hirsch and Yor~~\cite{Hirsch-Yor}. The exponential functional $I$  of $\xi$ is defined by
\[
I:=\int^{\infty}_{0}\exp\{-\xi_{s}\}\, ds.
 \]
 For background on exponential functionals and related objects, see the thorough review by Bertoin and Yor \cite{Bertoin-yor-2005}. The random variable $I$ is finite almost surely because either the life-time of $\xi$ is finite a.s. or the strong law of large numbers for  subordinators ensures that it grows at least linearly. Following \cite{Bertoin-yor-2005}, the distribution of $I$ is determined by its
integer moments,  which, as proved by Carmona et al.$\,$\cite{CPY94}, satisfy
 \[
 \er [I^{n}]=\prod^{n}_{i=1}\frac{i}{\phi(i)},\quad n\geq 1.
 \]
We also refer to Gnedin et al.$\,$\cite{GPY} for combinatorial derivations of the moments formula.  Bertoin and Yor \cite{BY2001} proved the  existence of a  r.v.$\,$$R,$  to which Hirsch and Yor~\cite{Hirsch-Yor} gave the name of {\it remainder random variable associated to $I$}, and whose integer moments are given by
\[
\er  [R^{n}]=\prod^{n}_{i=1}\phi(i),\quad n\geq 1.
\]
Furthermore, the distribution of $R$ is moment determinate and, if
$I$ and  $R$ are  taken to be independent then
 \begin{equation}\label{factorisation}
 I R\stackrel{\text{(d)}}{=}\erv,
 \end{equation}
 where $\stackrel{\text{(d)}}{=}$ means equality in distribution and
 $\erv$ is a standard exponential random variable. It has been proved by Berg \cite{berg2}, Theorem 2.2, that $\log{R}$ is a spectrally negative infinitely divisible (i.d.$\,$for short) random variable. Furthermore, its Laplace exponent which, in this case, is defined by
$$\psi(\lambda):=\log \er \left[\exp\{\lambda \log{R}\}\right],$$ is specified by
\begin{equation}\label{PsiR}
\psi(\lambda)=\lambda\log{\phi(1)}+\int_{(0,\infty)}\left((e^{-\lambda
x}-1)+\lambda(1-e^{-x})\right)\frac{e^{-x}}{1-e^{-x}}\frac{\kappa(dx)}{x},
\quad \lambda\geq 0,
\end{equation}
where $\kappa(dx)$ is the unique measure whose Laplace transform is given by
\begin{equation}\label{kappa}
\int_{\mathbb{R}^{+}}\kappa(dx)e^{-\lambda x}=\frac{\phi'(\lambda)}{\phi(\lambda)},\quad \lambda >0.
 \end{equation}
Notice that the rightmost term in (\ref{kappa}) is, indeed, the Laplace transform of a nonnegative measure because $\phi^{\prime}$ and $1/\phi$ are both completely monotone functions, and the product of c.m.$\,$functions is a c.m.$\,$function. The former is the Laplace transform of the measure $a\delta_{0}(dx)+{\bf 1}_{\{x>0\}}x\Pi(dx),$ and the latter is that of the potential of the subordinator $\xi$, see the forthcoming identity (\ref{potential}).
The L\'evy measure of the distribution of $\log{R}$ is then the image by the map $x\rightarrow -x$ of the measure given on the positive half-line by $x^{-1}(e^{x}-1)^{-1}\kappa(dx)$, see \cite{Hirsch-Yor} for further details and a proof of this fact.

In this paper, we are primarily interested in finding a necessary and sufficient condition for the r.v.$\,$$\log I$ to be i.d. This question has been studied by Berg \cite{berg1}, Hirsch and Yor~\cite{Hirsch-Yor}, and   Urbanik \cite{Urbanik} who obtained necessary and sufficient conditions for this to hold in terms of $\phi$ or the measure $\kappa.$ Our first main result, Theorem \ref{exp-undershoot-id}, provides a further equivalent condition which has the advantage of being stated directly in terms of the L\'evy measure of $\xi$. In  Theorem \ref{distribution-infty}, we provide an expression for the moment generating functions of $\log I$ and $\log R$ in the form  of infinite products involving  $\phi.$   It is worth mentioning that the infinite products complete the representations obtained by Maulik and Zwart \cite{Maulik-Zwart-06} and that a closely related result appeared recently in the paper by Patie and Savov \cite{patie-savov}.
Finally, in Theorem \ref{R-product} we obtain some identities in law for $R,$ which extend known identities for the gamma r.v.$\,$obtained by
Gordon~\cite{gordon}.  A key object in our work is the so-called harmonic potential measure of a subordinator which is defined below; we will emphasize its importance in the study of subordinators. In order to provide more details, we will introduce further notation and state some preliminary facts.

To start with, the lifetime $\zeta$ of the subordinator $\xi$ is either  $+\infty$ a.s., in which case we say that $\xi$ is immortal, or exponentially distributed with parameter $\phi(0)=q$, in which case $\xi$  jumps  to the cemetery state $\Delta$ at time $\zeta$ and stays there forever.  In any case, $\xi$ can always be seen as an immortal subordinator killed at an independent exponential time with parameter $q\geq0$, where the case $q=0$ is included to permit $\zeta=\infty$ a.s.   The representation of the Laplace exponent $\phi$ in (\ref{LK}) states that $\phi$ is a Bernstein function. So, associated to $\xi$ there is a Bernstein function. Conversely, given a Bernstein function $\phi$ it is well known that the function $\exp\{-\phi(\lambda)\}$ is the Laplace transform of an i.d.$\,$probability measure, which is concentrated on $\re^{+},$ and thus, associated to it there is a subordinator. For background on subordinators and  Bernstein functions, see for instance \cite{Bertoin-1999} and~\cite{SSV}, respectively.

Next, we denote by $V(dx)$ the potential measure (p.m.$\,$for short) of $\xi$, that is the measure
\[
V(dx)=\int_{\re^{+}}dt\pr(\xi_{t}\in dx),\quad x\geq 0,
\]
 which is characterized by the Laplace transform
\begin{equation}\label{potential}
\int_{\re^{+}}V(dx)\, e^{{-\lambda x}}=1/\phi(\lambda),\quad \lambda\geq 0.
 \end{equation}
The p.m. $\,$$V$  characterizes the law of $\xi.$  For  $\alpha\geq 0,$ the $\alpha$-resolvent of $\xi,$ say $V_{\alpha},$ is the measure given by
 \[
 V_{\alpha}(dx)=\int^{\infty}_{0}dt \,e^{-\alpha t}\pr(\xi_{t}\in dx), \quad x\geq 0.
  \]
  Observe that $V_{0}=V,$ so, hereafter, the latter and former quantities will refer to the same object.

 As we mentioned above, one of our objectives  is to emphasize the importance of the so-called {\it harmonic potential measure} (h.p.m.$\,$for short) of a subordinator, particularly, in the context of the characterization of L\'evy measures of some i.d.$\,$distributions. That is,  distributions of random variables which are related to exponential functionals of subordinators and undershoots. The h.p.m.$\,$is the measure defined by
\[
H(dx):=\int_{0}^{\infty}\frac{dt}{t}\pr(\xi_{t}\in dx),\quad x> 0.
\]This terminology  is adopted from the corresponding object in random walks theory, i.e. harmonic renewal measure,  see for instance \cite{GOT}. As we will see later, in the examples, there are several subordinators for which the h.p.m.$\,$$H$ can be obtained explicitly. More can be extracted from references such as \cite{Hirsch-Yor} and \cite{Song-Vondracek-2006}.

 In the next Lemma, we establish an identity relating the h.p.m.$\,$to the L\'evy measure  of the i.d.$\,$distribution giving the last position of $\xi$ prior to crossing. To that end,   let  $L_{t}=\inf\left\{s>0: \xi_{s}>t\right\}$ be the first passage time above the level $t$, for $\xi,$ and $G_{t}=\xi_{L_{t}-}$   be the last position below the level $t$. We define the L\'evy tail   $\overline{\Pi}: \mathbb{R}^+\rightarrow \mathbb{R}$  by   setting  $\overline{\Pi}(x)=q+\Pi(x, \infty)$, for  $x> 0$, and $\overline{\Pi}(+\infty)=q$.
\begin{lemma}\label{exp-undershoot}
Assuming that  $\erva$ is an  exponential random variable with parameter $\alpha>0$, which is independent of $\xi$, the following assertions hold true.
\begin{itemize}
\item[1)] The last position below the level ${\erva}$ and the undershoot, $G_{\erva}$ and $\erva  - G_{\erva}$ respectively, are independent r.v.'s and their Laplace transforms are given,  for every $\lambda\geq 0$, by
\begin{eqnarray}
 \er \left[\exp\{-\lambda \, G_{\erva}\}\right]&=&\frac{\phi(\alpha)}{\phi(\alpha+\lambda)} \label{undershoot1}
 \end{eqnarray}
 and
 \begin{eqnarray}
\er \left[\exp\{-\lambda \, (\erva  - G_{\erva})\}\right]&=& \frac{\alpha}{\phi(\alpha)} \frac{\phi(\alpha+\lambda)}{\alpha +\lambda}.\label{undershoot0}
 \end{eqnarray}
 \item[2)]  $G_{\erva}$  is a positive i.d.$\,$random variable having the probability distribution
 \[
 \pr(G_{\erva}\in dx)=\phi(\alpha)e^{-\alpha x}V(dx),\quad x\geq 0.
 \]
 Its Laplace exponent is given by the Bernstein function
 $$\lambda\mapsto \log \frac{\phi(\lambda +\alpha)}{\phi(\alpha)} = \int_{(0,\infty)}  (1- e^{-\lambda x}) e^{-\alpha x} \, H(dx),\quad \lambda \geq 0.$$
 In particular, we have
 \[
 \er\left(G_{\erva}\right)=\frac{\phi^{\prime}(\alpha)}{\phi(\alpha)}<\infty.
 \]
\item[3)] $\erva  - G_{\erva}$ has the distribution
\[
\pr( \erva  - G_{\erva}\in dx)=
\frac{\alpha }{\phi(\alpha)} \left(a \delta_0(dx)  + e^{-\alpha x} \overline{\Pi}(x)dx\right),\quad x\geq 0.
\]
\end{itemize}
\end{lemma}
 Formula (\ref{undershoot1}) can be found in Bertoin \cite{Bertoin-1999} Lemma 1.11, we also refer to Winkel \cite{Winkel} for the study of some related distributions.  For the sake of completeness, a proof will be given below in Section \ref{subsectionundershoot}.

\begin{remark} In the case where $x\mapsto  \overline{\Pi}(x)$ is log-convex, which is equivalent to say that the function
$x\mapsto \alpha e^{-\alpha x} \overline{\Pi} (x)/\phi(\alpha)$
bears the same property for some (and hence for all) $\alpha >0$,
Steutel's theorem  (Theorem 10.2 in Chapter III in Steutel and van
Harn \cite{steutel}) insures then that the undershoot $\erva-
G_{\erva}$ is i.d. This fact is  relevant for
Theorem \ref{exp-undershoot-id}  which is stated below.
\end{remark}
Before stating our first main result, we provide an elementary, but key, lemma that allows us to relate the h.p.m.$\,$with the measure $\kappa$ defined by formula (\ref{kappa});  this result plays an important role in the subsequent characterization of the law of $R$ and  the infinite divisibility of $\log I.$
This will also lead to a relation between these random variables and,  $G_{\erva}$ and $\erva  - G_{\erva}$.
\begin{lemma}\label{harmonic}We have the following identity,
\[
\iint_{(0,\infty)\times (0,\infty)}ye^{-\lambda y}\pr(\xi_{t}\in dy)\frac{dt}{t}=\frac{\phi'(\lambda)}{\phi(\lambda)},\quad \lambda>0.
 \]
It follows that
 \[
 \kappa(dx)=xH(dx),\quad x>0.
 \]
Furthermore, $H(dx)$ is the unique measure such that
\begin{equation}\label{fisurfi}
\log \frac{\phi(\lambda)}{\phi(1)}=\int_{(0,\infty)}
(e^{-x}-e^{-\lambda x})H(dx), \quad \lambda >0.
\end{equation}
 As a consequence, for any fixed $c>0,$ the measure $e^{-cx}H(dx)$ is the h.p.m.$\,$of the subordinator with Laplace exponent $\phi(\cdot+c)$  and the L\'evy measure of the r.v.$\,$$G_{\erv_c}$.
\end{lemma}

\begin{remark} For $q\in \re$, let us define  a measure $\mathcal{H}_{q}$
by
 \[
 \mathcal{H}_{q}(dx):=\int^{\infty}_{0}\frac{dt}{t} \,e^{-q  t}\pr(\xi_{t}\in dx), \quad x> 0.
  \]
When $q\geq 0$, $\mathcal{H}_{q}$ corresponds to
the harmonic $q$-potential measure associated to $\xi$. It is simply
the h.p.m of the subordinator $\xi$ killed at rate $q$, which has the Laplace exponent  $\lambda \mapsto q+ \phi(\lambda)$ . When $q <
0$, the measure $\mathcal{H}_{q}$ appears, for example, if we look at the
 h.p.m.$\,H^{(\alpha)}$ associated to the subordinator
$\xi^{(\alpha)}$ with Laplace exponent $\lambda \mapsto
\phi(\lambda+\alpha)-\phi(\alpha)$ for some $\alpha >0$. Indeed,  since
$$\pr( \xi^{(\alpha)}_t\in dx)=\exp \{-\alpha x +\phi(\alpha)t \} \pr( \xi_t \in dx),\quad t\geq 0,$$
we obtain $H^{(\alpha)}(dx)=e^{-\alpha
x}\mathcal{H}_{-\phi(\alpha)}(dx)$, $x>0$. Of
course, for $\alpha>0$, $H^{(\alpha)}$ is different from the
h.p.m.$\,$$e^{-\alpha \cdot }H$  which corresponds to the Bernstein
function $\lambda \mapsto \phi(\lambda+\alpha)$ of
Lemma~\ref{harmonic} above.
\end{remark}

The focus in the following Theorem is on a characterization of the infinite divisibility of the r.v.'s $\erva- G_{\erva}$ and $\log I$ in terms of $\phi$, $H$ and $\Pi$.
\begin{teo}\label{exp-undershoot-id}
 The following assertions are equivalent:
\begin{itemize}
 \item[(i)] $\erva- G_{\erva}$ is i.d.$\,$for some, and hence for all,  $\alpha >0$;

 \item[(ii)]  the probability measure
\[
\frac{\alpha }{\phi(\alpha)} \left(a \delta_0(dx)  + e^{-\alpha x} \overline{\Pi}(x){\bf 1}_{\{x>0\}} dx\right),\quad x\geq 0,
 \]
 is the law of an i.d.$\,$random variable for some, and hence for all,  $\alpha>0$;

 \item[(iii)]  $\log I$ is infinitely divisible;

 \item[(iv)] $\lambda \mapsto \frac{1}{\lambda} - \frac{\phi'(\lambda)}{\phi(\lambda)}$ is completely monotone (c.m.$\,$for short);

 \item[(v)] the measure $dx - x H(dx), \, x>0,$ is nonnegative;

\item[(vi)] the h.p.m.$\,$of $\xi$ has a density on $(0,\infty),$ say $\rho,$ with respect to the measure $dx/x$ which is such that $\rho(x)\leq 1$ for every $x>0.$
 \end{itemize}
\end{teo}
Notice that a consequence of  Lemma~\ref{harmonic} is that
condition $(v)$  is equivalent to
\begin{itemize}
\item[$(v')$] The measure $dx - \kappa(dx), \, x>0,$ is nonnegative.
\end{itemize}
A portion of Theorem \ref{exp-undershoot-id} is known. The equivalence between (iii) and (iv) has been proved in \cite{berg1}, and the equivalence between (iv) and (v') is obtained in \cite{Hirsch-Yor}. The equivalence between (v) and (vi) is elementary. Our main contribution to this result is the equivalence between (ii) and (v) and thus to (v'). Condition (ii) has the advantage of being stated in terms of the L{\'e}vy measure.  Although verifying the infinite divisibility condition could be a hard task,  the following corollary gives a straightforward consequence which does not seem to be easy to obtain from the other conditions in the theorem.
\begin{cor} If $\xi$ has no killing term and its jumps are bounded above by some fixed and finite constant $c>0,$ equivalently $\overline{\Pi}(x)=0$ for $x>c,$ then the r.v.$\,$$\log I$ is not infinitely divisible.
\end{cor}
This corollary readily follows from (ii) and the well known fact that an i.d.$\,$distribution can not have a bounded support.

Moreover, relating the measures $\kappa$ and $H$ allows to have a
better understanding of $\kappa$. See
Section~\ref{furtherproperties} where several properties of $H$ are
obtained; these will give place to the proofs
of the forthcoming Corollaries~\ref{cor:1} and \ref{cor:2}.
\begin{cor}\label{cor:1}
If  $\xi$ has a strictly positive drift and is not a degenerate pure drift process then the random variable $\log{I}$ is not i.d.
\end{cor}
The proof of this Corollary is deferred to page \pageref{corollary002}  because it  uses  condition {\it (v)} of Theorem~\ref{exp-undershoot-id} together with the properties of the h.p.m.$\,$$H$ established in Section~\ref{furtherproperties}.

Our next goal is to give a formula which describes the moments of $I$ and $R$. A related result, which deals with a larger class of L\'evy processes, appeared in Patie and Savov \cite{patie-savov}. However, our approach for subordinators is based on a study by Webster on the so-called generalized gamma functions.

We say  that a function $g : (0, \infty) \to (0, \infty)$  satisfies Webster's conditions if
\begin{equation}\label{web}
\mbox{\it $g$ is log-concave $\quad$ and $\quad \lim_{s\rightarrow \infty} \frac{g(s+c)}{g(s)}=1, \;$ for every $c>0$}
\end{equation}
 Following Webster \cite{Webster-1997}, we set  $a_n= (g'_{-}(n)+g'_{+}(n))/2g(n)$, $n\geq 1,$
 \begin{equation}\label{EM-gen-gamma}
 \gamma_g= \lim_{n\rightarrow \infty} \left(\sum_1^{n}
a_j-\log g(n) \right),
\end{equation}
and define the  generalized gamma function associated to $g$  by
\begin{equation}\label{gamma-1}
\Gamma_g(s):=\frac{e^{-\gamma_g s}}{g(s)}
\prod_{n=1}^{\infty}\frac{g(n)}{g(n+s)}e^{a_n s}, \quad s>0.
\end{equation}
Recall that the exponential functional stopped at time $t>0$  is defined by
\[
I_{t}=\int^{t}_{0}\exp\{-\xi_{s}\}\,ds.
\]
\begin{teo} \label{distribution-infty} The following assertions hold true.
 \begin{itemize}
 \item[(1)]  The functions $\lambda \mapsto \phi(\lambda)$ and $\lambda \mapsto \phi^* (\lambda):= \lambda / \phi(\lambda)$ satisfy  (\ref{web}) and the moment generating functions of $\log R$ and $\log I$ are given by
  \[
  \er [R^s]=\Gamma_\phi(s+1)\quad  and \quad \er [I^{s}]=\Gamma_{\phi^*}(s+1)=\frac{\Gamma(s+1)}{\Gamma_\phi(s+1)}, \quad  s>-1 .
  \]
\item[(2)]  We assume here that $q=\phi(0)=0$.  Let $\alpha>0$, $\erva$  be an exponentially distributed random variable with parameter $\alpha$ independent of $\xi$ and denote $\phi_{c,\alpha}(\cdot) =\phi(\cdot+c)+\alpha,\,c \geq 0$. Then, the joint distribution of $(I_{\erva}, \xi_{\erva})$ is characterized by
\begin{equation}\label{IjointXi}
\er [(I_{\erva})^s e^{-\mu
\xi_{\erva}}]=\frac{\alpha}{\alpha+\phi(\mu)}\,\frac{\Gamma(s+1)}{\Gamma_{\phi_{\mu,\alpha}}(s+1)}, \quad \mu \geq 0, \,s>-1.
\end{equation}
\end{itemize}
\end{teo}

As we mentioned before, C. Berg proved in \cite{berg2} that the r.v.$\,$$\log{R}$ is i.d.$\,$and established formula (\ref{PsiR}) for its characteristic exponent.
By  Lemma~\ref{harmonic}, we can replace $\kappa(dx)$ by $xH(dx)$ in there. This yields the first statement in the following theorem. We have chosen to include this fact in the theorem because the description of the L\'evy measure of $\log R$ in terms of the h.p.m.$\,$$H$ has the advantage of giving  succinct information about the former measure, as we already pointed out. For instance, when $\xi$ is arithmetic, that is, its support is a subset of a subgroup $\mathcal{C}$ of $k\mathbb{Z}^{+}$ for some $k$, then $H$ is also carried by $\mathcal{C},$ and hence the L\'evy measure of $\log R$ is carried by $-\mathcal{C}.$ In the following theorem, we provide an identity in law describing $R$ as an infinite product of independent last positions of $\xi$ below random barriers. This result and Corollary~\ref{remainder-gordon} can be seen as a generalization of Gordon's representations of a log-gamma random variable involving a sequence of independent standard exponential random variables, see Gordon \cite{gordon} and  Example~\ref{example-1} of Section \ref{section:examples} below for further details.

Throughout this paper, for a real valued r.v. $A$ we will refer to the function $\lambda\mapsto \er(e^{\lambda A}),$ in the domain where this is finite, as the Laplace transform of $A.$ Consequently, we will refer to $\pm\log \er(e^{\lambda A})$  as the Laplace exponent of $A;$ the choice of the sign will be made clear in each case. In the case where $A$ is i.d.$\,$the Laplace exponent is obtained by analytical extension of its characteristic exponent.

\begin{teo}\label{R-product}
 The Laplace  exponent of $\log{R}$, i.e. the function $\lambda \mapsto \log \er [R^\lambda]$, is given by
\begin{equation} \label{loggc}
\log\Gamma_\phi (\lambda + 1)=  -\lambda\gamma_{\phi} + \int_{(0,\infty)}\left(e^{-\lambda x}-1 + \lambda x\right)\frac{e^{-x}}{1-e^{-x}}H(dx).
\end{equation}
Its L\'evy measure is the image by the map $x\rightarrow -x$ of the measure  ${\bf 1}_{\{x>0\}}(e^{x}-1)^{-1}H(dx)$.  Furthermore,
we have the following equalities
\begin{equation}\label{factor-0}
\log{R}\stackrel{\text{(d)}}{=}\log{\phi(1)}+\sum_{k=1}^{\infty}\left(\phi_{k}(1)-G^{(k)}\right)=-\gamma_{\phi}+\sum_{k=1}^{\infty}\left(\er [G^{(k)}]-G^{(k)}\right),
\end{equation}
where  $(G^{(k)}, k\geq 1)$ are independent random variables such that  $G^{(k)}$, for $k=1,2,\cdots$,  is a copy of $G_{\ervk}$,  $\er [G^{(k)}]= \frac{\phi'(k)}{\phi(k)}$ and
\[
\phi_{k}(1)=-\log \er \left[e^{-G^{(k)}}\right]=  \int_{(0,\infty)}(1-e^{-x}) \,e^{-kx}\,H(dx)=\log \frac{\phi(k+1)}{\phi(k)}.
\]

 \end{teo}
 The following result gives the analogue of (\ref{factor-0}) in which the infinite series is replaced by a partial sum.
\begin{cor}\label{remainder-gordon} Keeping the notation of Theorem \ref{R-product}, we have the representation
\begin{equation*}\label{Gordon-2cor}
\begin{split}
\log{R}&\stackrel{\text{(d)}}{=}-\gamma_{\phi}+\sum_{k=1}^{n}\left(\er [G^{(k)}]-G^{(k)}\right)+\log\frac{R_{(n)}}{\phi(n+1)}+B_n,
\end{split}
\end{equation*}
where $R_{(n)}$, which is the remainder random variable associated to the subordinator with Laplace exponent $\lambda\mapsto \phi(\lambda+n),$ $\lambda\geq 0,$  is independent of $(G^{(k)}, 1\leq k\leq n),$ and  is such that $\log\frac{R_{(n)}}{\phi(n+1)}\to 0$ in probability; and finally
\begin{equation*}
\begin{split}
B_n&=\int_{(0,\infty)}H(dx)\frac{e^{-x}}{(1-e^{-x})}\left(e^{-x}-1+x\right)e^{-nx}\xrightarrow[n\to \infty]{}0.
\end{split}
\end{equation*}
\end{cor}
As a consequence of the first assertion of Theorem \ref{R-product} and properties of the h.p.m.'s, we state the following result.
\begin{cor}\label{cor:2}
If  the L\'evy tail $\overline{\Pi}$  is log-convex, then the random variable $\log{R}$ belongs to class of self-decomposable distributions; if furthermore $\overline{\Pi}$ is c.m.$\,$then $\log{R}$ belongs to the class of extended generalized gamma convolutions.
\end{cor}
We refer to the books of Bondesson~\cite{bondesson}  and  Sato \cite{sato} for background on self-decomposable and extended generalized convolutions distributions.

The rest of this paper is organized as follows. Section 2 is devoted
to some preliminary results and proofs. To be more precise, in Subsection \ref{preliminaries} we
recall some facts on subordinators, then we prove
Lemmas~\ref{exp-undershoot} and \ref{harmonic}  in Subsection
\ref{subsectionundershoot} and, finally, Theorem~\ref{exp-undershoot-id} is established in
Subsection \ref{subsection-exp-undershoot-id}.
Theorem~\ref{distribution-infty} and its Corollary are proved in
Subsection
\ref{subsection-distribution-infty},
 where some other
related results are recalled. Subsection~\ref{sect:remainders} is
devoted to the study of the distribution of the remainder random
variable. In particular, Theorem~\ref{R-product} and its Corollary
are proved therein. In Section~\ref{furtherproperties}, we obtain
some basic results about harmonic potential measures and prove
Corollaries~\ref{cor:1} and \ref{cor:2}. To finish, some examples
are studied in Section~\ref{section:examples}.
\section{Preliminaries and proofs}\label{sect2}
\subsection{Preliminaries}\label{preliminaries}
As for the p.m.,  the h.p.m.$\,$is also related to $\phi$ via Laplace
transforms but in a more involved way. Indeed, by Frullani's  integral, we have the
following identities: for $\lambda\geq 0,$ $\alpha>0$,
\begin{equation}\label{eq:1}
\begin{split}
\frac{\phi(\alpha)}{\phi(\alpha+\lambda)}&=\exp\left\{\log\left(\frac{\phi(\alpha)}{\phi(\alpha+\lambda)}\right)\right\}\\
&=\exp\left\{\int_{0}^{\infty}\frac{dt}{t}\left(e^{-t\phi(\alpha+\lambda)}-e^{-t\phi(\alpha)}\right)\right\}\\
&=\exp\left\{\int_{0}^{\infty}\frac{dt}{t}\left(\er [e^{-(\lambda+\alpha)\xi_{t}}]-\er [e^{-\alpha\xi_{t}})\right]\right\}\\
&=\exp\left\{-\int_{0}^{\infty}\frac{dt}{t}\int_{\re^{+}}\pr(\xi_{t}\in dx)e^{-\alpha x}(1-e^{-\lambda x})\right\}\\
&=\exp\left\{-\int_{(0,\infty)}H(dx)e^{-\alpha x}(1-e^{-\lambda x})\right\}.
\end{split}
\end{equation}

Now, let us observe that  the Bernstein function $\phi$, specified by (\ref{LK}), has the representation
\[
\phi(\lambda)=\lambda a+\int_{(0,\infty]}(1-e^{-\lambda x})\left(\Pi(dx)+q\delta_{\infty}(dx)\right), \quad \lambda\geq 0.
\]
 This allow us to extend the L\'evy-It\^o
representation of $\xi$ as
\[
\xi_{t}=at+\sum_{s\leq t}\Delta_{s}, \quad t\geq 0,
 \]
 where $((t,\Delta_{t}), t\geq 0)$ forms a Poisson point process on $(0,\infty)\times(0,\infty]$ with intensity measure $ds\otimes (\Pi(dx)+q\delta_{\infty}(dx)).$ Said otherwise, the killing of the subordinator arises when there appears a jump of infinite size. A different way to view a  subordinator $\xi$ is by means of an immortal subordinator, that is one with an infinite lifetime, whose characteristic triplet is $(0,a,\Pi)$, say $\xi^{\dagger},$ killed at  time $\erv_q,$ where $\erv_q$ is an independent exponential random variable with parameter $q>0$; {that is to say by killing $\xi^{\dagger}$ at rate $q$.} More precisely, the process defined by

\begin{equation}\label{tue}
\xi^{\dagger, q}_{t}:=\begin{cases}\xi^{\dagger}_{t}, &\text{if} \ t<\erv_q;\\
\infty, & \text{otherwise,}
\end{cases}
 \end{equation}
is a subordinator with characteristic triplet $(q,a,\Pi),$ and hence with the same law as $\xi.$ This means that the transition probabilities  of $\xi^{\dagger}$ are related to those of $\xi$ by an exponential factor  viz.
\[
\pr(\xi_{t}\in dx)=e^{-qt}\pr(\xi^{\dagger}_{t}\in dx),\ \ \text{on} \ [0,\infty),\ \text{for all}\ t\geq 0.
\]
 It is also easily seen that if we kill at a rate $\alpha$ an already killed subordinator with characteristic triplet $(q,a,\Pi)$ then we obtain again a killed subordinator with characteristics $(q+\alpha,a,\Pi)$. These elementary observations have as a consequence that the  $(\alpha+q)$-resolvent  of $\xi^{\dagger}$ equals the potential of the subordinator with characteristic triplet $(\alpha+q,a,\Pi)$ or, equivalently, the $\alpha$-resolvent of the subordinator with characteristic triplet $(q,a,\Pi).$
 These remarks will be useful in the following proof and in Section \ref{furtherproperties}.
\subsection{Proof of Lemmas~\ref{exp-undershoot} and
\ref{harmonic}}\label{subsectionundershoot}  We have now all the
elements to give a proof to Lemma~\ref{exp-undershoot}.
\begin{proof}[Proof of Lemma~\ref{exp-undershoot}]
The proof of this result follows from the identities
\begin{equation}\label{volterraeq}
\pr(G_{t}\in dz)=V(dz)\overline{\Pi}(t-z),\quad 0\leq z< t,
\end{equation}
and
\begin{equation}\label{volterraeq-2}
\pr(G_{t}=t)=av(t),\quad t\geq 0,
\end{equation}
 where $v$ stands for the density of the p.m.$\,$$V,$ which we know exists whenever the drift $a>0.$ Relationships (\ref{volterraeq}) and (\ref{volterraeq-2})
 can essentially be read in Bertoin (\cite{Bertoin-1996}, Proposition III.2) where the result is proved for immortal subordinators. But, the same argument holds for killed subordinators.  Indeed, arguing as in the proof of the aforementioned proposition of \cite{Bertoin-1996}, it is proved that
 \[
 \pr(G_{t}\in dz)=V^{\dagger}_{q}(dz)\left(\Pi(t-z,\infty)+q\right) ,\quad 0\leq z< t,
 \]
where $V^{\dagger}_{q}$ denotes the $q$-potential of
$\xi^{\dagger},$ and the term $qV^{\dagger}_{q}(dz)$  comes from the
possibility that  the process passes above the level by jumping to
its cemetery state $\Delta$ , that is $L_{t}=\zeta$. Furthermore, to
verify that (\ref{volterraeq-2}) holds true, we use the formulae
 \begin{eqnarray*}
 \pr(G_{t}=t)&=&\pr(G^{\dagger}_{t}=t, L_{t}<\erv_q)\\
 &=&\er \left[{\bf 1}_{\{\xi^{\dagger}_{L_{t}-}=t\}}e^{-qL_{t}}\right]\\
 &=&\er [e^{-qT^{\dagger}_{\{t\}}}{\bf 1}_{\{T^{\dagger}_{\{t\}}<\infty\}}]\\
 &=&av^{\dagger}_{q}(t),\quad t>0.
  \end{eqnarray*}
  Here,  $T^{\dagger}_{\{t\}}=\inf\{s>0: \xi^{\dagger}_{s}=t\}$  and $v^{\dagger}_{q}$ denotes the density of the measure $V^{\dagger}_{q}$, which is known to exist whenever the drift is strictly positive, and the last equality follows from the equality in the bottom of page 80 of \cite{Bertoin-1996}. From the discussion  preceding  this proof, we have that $v=v^{\dagger}_{q}$.   In the remainder of this proof, we use
 \begin{equation}\label{LK-2}
\frac{\phi(\lambda)}{\lambda}= a+\int_{0}^{\infty}e^{-\lambda x}\overline{\Pi}(x)\, dx.
 \end{equation}
 which is obtained by integrating by parts (\ref{LK}).

1) Taking joint Laplace   transforms, we get that for all $\lambda,\mu \geq 0$,
\begin{eqnarray*}
\Psi(\lambda,\mu)&=&\er \left[\exp\{-\lambda G_{\erva} - \mu (\erva-G_{\erva}) \}\right]\\
&=&\alpha\int^{\infty}_{0}dt \, e^{-(\mu +\alpha) t }\int_{[0,t]}\pr(G_{t}\in dz)e^{(\mu -\lambda) z}\\
&=&\alpha\int^{\infty}_{0}dt \,e^{-(\mu +\alpha) t}\int_{[0,t)}V(dz)\overline{\Pi}(t-z)\,e^{(\mu -\lambda) z}+\alpha\int^{\infty}_{0}dt \,e^{-(\alpha+\lambda) t}a v(t)\\
&=&\alpha\int_{[0,\infty)}V(dz)\,e^{(\mu -\lambda) z} \,e^{-(\mu +\alpha) z} \int^{\infty}_{z}dt\, e^{-(\mu +\alpha) (t-z)}\,\overline{\Pi}(t-z)+\frac{\alpha a}{\phi(\alpha+\lambda)}\\
&=&\frac{\alpha}{\phi(\alpha+\lambda)}\left(a + \int^{\infty}_{0}du
\, e^{-(\mu +\alpha) u}\,\overline{\Pi}(u)\right).
\end{eqnarray*}
It follows, by making use of (\ref{LK-2}), that
\begin{eqnarray*}
\Psi(\lambda,u)=\frac{\alpha}{\phi(\alpha+\lambda)} \frac{\phi(\alpha+ \mu)}{\alpha+\mu}
=\frac{\phi(\alpha)}{\phi(\alpha+\lambda)} \times \frac{\alpha \, \phi(\alpha +\mu)}{\phi(\alpha)\,(\alpha +\mu)}
\end{eqnarray*}
which gives the independence.

2)   The infinite divisibility property of $G_{\erva}$  follows from (\ref{eq:1}) and the fact that
\begin{eqnarray*}
\int_{\mathbb{R}^+}(1\wedge x) e^{-\alpha x}\int_{(0,\infty)}\pr (\xi_t\in dx)\frac{dt}{t}&=&\int_0^{\infty} \frac{dt}{t}\er [(1\wedge \xi_t) e^{-\alpha \xi_t}]dt\\
&\leq& \int_0^{\infty} \frac{dt}{t}\er [\xi_t e^{-\alpha \xi_t}]dt<\infty.
\end{eqnarray*}
The finiteness in the last inequality is obtained from
\begin{equation}\label{calcul}
\int^{\infty}_{0}\frac{dt}{t}\er [\xi_{t}\, e^{-\lambda\xi_{t}}]=\int^{\infty}_{0}
\frac{dt}{t}\,t\,\phi'(\lambda)\, e^{-t\phi(\lambda)}=\frac{\phi'(\lambda)}{\phi(\lambda)}
\end{equation}
since
 \[
 \er [\xi_{t}\, e^{-\lambda \xi_{t}}]=-\left(\er [e^{-\lambda\xi_{t}}]\right)'=t\phi'(\lambda)\, e^{-t\phi(\lambda)},\quad \lambda,t>0.
  \]

 The expression giving the probability distribution $\pr(G_{\erva}\in dx)$ is obtained from (\ref{potential}). Finally, we have
 \[\er [G_{\alpha}]=\phi(\alpha)\int_{(0,\infty)}xe^{-\alpha x}V(dx)=\frac{\phi'(\alpha)}{\phi(\alpha)} \]
 where the  integral is evaluated by differentiating (\ref{potential}).  \\

3) From the above identities, it is immediate that we have
\begin{eqnarray}\label{eq:overshoot}
\er \left[\exp\left\{-\lambda\left(\erva- G_{\erva}\right)\right\}\right]&=&\frac{\alpha}{\phi(\alpha)}\frac{\phi(\alpha+\lambda)}{\alpha+\lambda}\nonumber\\
&=&\frac{\alpha}{\phi(\alpha)}\int_0^\infty e^{-\lambda x} \left[e^{-\alpha x} \overline{\Pi}(x)\, dx+a\delta_{0}(dx)\right]
\end{eqnarray}
where we used (\ref{LK-2}) to get the second equality.
\end{proof}

We next prove  Lemma~\ref{harmonic} which relates the measure $\kappa$, as defined in (\ref{kappa}), to the h.p.m.$\,$of $\xi$.
\begin{proof}[Proof of Lemma~\ref{harmonic}] The first statement of Lemma \ref{harmonic}  is a reformulation of (\ref{calcul}).
  The identity relating $\kappa$ and the h.p.m.$\,$of $\xi$ follows by inverting Laplace transforms. The last formula in Lemma~\ref{harmonic} is obtained by integrating (\ref{kappa}) over $[1, \lambda]$.
\end{proof}

\subsection{Proof of Theorem~\ref{exp-undershoot-id} } \label{subsection-exp-undershoot-id}
The equivalence between (iii) and (iv) is proved in Theorem 1.9 in Berg \cite{berg1}. The equivalence between (iv) and (v) is obtained from Theorem 3.3 in \cite{Hirsch-Yor} and the identification of the measure $\kappa$ in  Lemma~\ref{harmonic}. The equivalence between (v) and (vi) is straightforward. Finally, the equivalence between the assertions (v) and (i) is deduced from the identity
\[
\er \left[\exp\{-\lambda \, (\erva  - G_{\erva})\}\right]=\exp\left\{-\int^{\infty}_{0}\left(1-e^{-\lambda x}\right)e^{-\alpha x}\left( x^{-1}dx-H(dx)\right)\right\},\quad \lambda\geq 0,
\]
which follows from (\ref{undershoot0}), (\ref{eq:1}) and a further application of Frullani's integral. That (i) and (ii) are equivalent follows from the third assertion in Lemma~\ref{exp-undershoot}. Finally, that the statements in (i) and (ii) hold for all $\alpha>0$, when they hold for some $\alpha_0>0$, follows from the fact that,  if the measure $e^{-\alpha x}(x^{-1}dx-H(dx))$
is nonnegative for  $\alpha=\alpha_0,$ then it is also nonnegative for all $\alpha>0$.
\subsection{Proof of Theorem \ref{distribution-infty}}\label{subsection-distribution-infty}
Following \cite{Carmona-Petit-Yor-1997}, we start by observing that
the following functional equation for the moments of $I$ holds
\begin{equation}\label{functional}
\er [I ^{s}]=\frac{s}{\phi(s)}
\er [I ^{s-1}],\quad    s>0, \quad \hbox{and}\quad
\er [I^0]=1.
\end{equation}
By using (\ref{factorisation}), we get also a functional equation for the moments of $R$. That is
\begin{equation}\label{functional-1}
\er [R ^{s}]=\phi(s)
\er [R ^{s-1}], \quad  s>0, \quad \hbox{and}\quad
\er [R^0]=1.
\end{equation}
Next, given a log-concave function $g:\mathbb{R}_+\rightarrow \mathbb{R}_+$, we will denote by $g'_{+}$ and
$g'_{-}$  the right and left derivatives of $g$,
respectively. Consider the more general functional equation
\begin{equation}\label{functional-equation}
f(s+1)=g(s)f(s),\quad  s>0, \quad \hbox{and}\quad f(1)=1.
\end{equation}
This equation was  studied by Webster in
\cite{Webster-1997};  he was  motivated therein by the study of
generalized gamma functions and their characterization by a
Bohr-Mollerup-Artin type theorem.

\begin{teo}\label{Webster} (Webster \cite{Webster-1997})  Assume that $g$ is log-concave and $ \lim_{s\rightarrow \infty}g(s+c)/g(s)=1$ for all $c>0$.   Then, there exists a unique  log-convex  solution $f:\mathbb{R}_+\rightarrow \mathbb{R}_+$  to the functional equation (\ref{functional-equation}) satisfying $f(1)=1,$ which is given by the generalized gamma function $\Gamma_g(s)$ specified by (\ref{gamma-1}). Furthermore, if $g(s)\rightarrow 1$, as
$s\rightarrow \infty$, then we have
\begin{equation}\label{gamma-2}
\Gamma_g(s)=\frac{1}{g(s)} \prod_{n=1}^{\infty}\frac{g(n)}{g(n+s)}.
\end{equation}
\end{teo}
Theorem \ref{Webster} is obtained from a combination of Theorem 4.1
of \cite{Webster-1997} and the subsequent discussion in Section 7
for the case when $g$ (resp. $f$) is log-concave (resp. log-convex)
on $\mathbb{R}_+$.  Note that if $g(x)=x$ then  (\ref{gamma-2})
gives the well known  infinite product representation of the gamma
function
\begin{equation}\label{gamma-weierstrass}
\Gamma(s)=\frac{1}{s}e^{-\gamma s}\prod_{n=1}^{\infty}
\left(1+\frac{s}{n}\right)^{-1}e^{s/n} \quad \hbox{for} \quad s>0
\end{equation}
where $\gamma$ stands for the Euler-Mascheroni constant. A combination  of the aforementioned results and  some arguments taken from the proof of Theorem 2.2 in \cite{Maulik-Zwart-06} leads to the first assertion of Theorem~\ref{distribution-infty}.
\begin{proof}[Proof of Theorem~\ref{distribution-infty}] (1) To prove the first assertion, we observe that the Frullani integral allows us to write
\begin{equation}\label{log-repr}
\log\frac{\phi(\lambda)}{\phi(1)}=\iint_{(0,\infty)\times (0,\infty)} (e^{-x}-e^{-\lambda x})P(\xi_t\in dx)\frac{dt}{t},\quad \lambda>0,
\end{equation}
which implies the log-concave property for $\phi$. The function $\phi^*$ is log-concave since by (\ref{LK-2}),
$1/\phi^*$ is c.m.$\,$and then log-convex.  The condition $\lim_{s \rightarrow \infty} \phi(s+c)/\phi(s)=1$, for all $c\geq 0$, is obtained as follows. By using formula (\ref{LK}), we clearly have
$$\phi(s+c)-\phi(s)=ca+\int_{(0,\infty)}(1-e^{-cx})e^{-sx}\Pi(dx), \quad s,c\geq 0.$$  It follows that
$$\frac{\phi(s+c)}{\phi(s)}=1+\frac{ca}{\phi(s)}+\frac{1}{\phi(s)}\int_{(0,\infty)}(1-e^{-cx})e^{-sx}\Pi(dx), \quad s,c\geq 0,$$
where the right hand side   goes to $1$ thanks to the dominated convergence theorem and the fact that, when $a>0,$ $\phi(s)\to \infty$ as $s\to\infty.$
 Theorem \ref{Webster} implies that (\ref{functional-equation}), with
$g=\phi$, has the unique log-convex solution given by
$\Gamma_\phi(s)$. We claim that the function $s\mapsto \er [R^s]$ is
log-convex. Indeed, for $a,b>0$ such that $a+b=1$ and $0\leq
u,v<\infty,$ we apply H\"older's inequality to infer the relation
$$\er\left[R^{au+bv}\right]=\er\left[R^{au}R^{bv}\right]\leq \left(\er\left[R^{u}\right]\right)^{a}\left(\er\left[R^{v}\right]\right)^{b}.$$ The log convexity follows.
Hence, $\er [R^{s}]=\Gamma_\phi(s+1)$. Now, the function
$h:s\rightarrow \er [R^{s}]\er [I^s]$
satisfies (\ref{functional-equation}) with $g(s)=s$ and $h(1)=1$.
Hence, by Bohr-Mollerup-Artin theorem for the gamma function, we
conclude that $h(s)=\Gamma(s)$. For the second statement, on the one hand, we have $\er [\erv^{s}]= \Gamma(s+1)$. On the other hand, using the factorization (\ref{factorisation}), we can write
\begin{eqnarray*}
\er [\erv^{s}]=\er [I^{s}]\er [R^{s}]=\er [I^{s}]\, \Gamma_\phi(s+1).
\end{eqnarray*}
The result follows.

(2)  Assuming that $\xi$ is immortal, let $Y$ be $\xi$ killed at
rate $\alpha$.  Plainly,  $Y$ is a subordinator with Laplace
exponent $\phi(\cdot)+\alpha$. Thus, we have
\[
I_{\erva}=\int_0^{\infty}e^{-Y_s}\, ds.
\]
For convenience, in the remainder of this proof, we replace the notation $\Gamma_\phi(s)$ by $\Gamma(\phi(\cdot); s)$. The first assertion of  Theorem \ref{distribution-infty}, when applied to $Y$, implies that
\[
\er [I_{\erva} ^{s}]=\frac{\Gamma(s+1)}{\Gamma(\phi(\cdot) +\alpha; s+1)}, \quad  s>0.
\]
Next, let us introduce a new probability measure by
setting
\begin{equation*}
d\pr_{|\mathcal{F}_t}^{(\mu)}=e^{-\mu
\xi_t+\phi(\mu)t}d\pr_{|\mathcal{F}_t}, \quad t\geq 0,
\end{equation*}
and denote by $\er^{(\mu)}$
the expectation under $\pr^{(\mu)}$. Clearly, under $\pr^{(\mu)}$,  $\xi$  is a
subordinator with Laplace exponent $\phi(\cdot+\mu)-\phi(\mu)$.
Thus, applying the first assertion,   we obtain
\begin{equation} \label{exponential-2}
\er^{(\mu)}[I_{\erva}^{s}]=\frac{\Gamma(s+1)}{\Gamma(\phi(\cdot+\mu)-\phi(\mu)+\alpha; s+1)}.
\end{equation}
Now, by performing a change of measure, we can write
\begin{eqnarray*}
\er^{(\mu)}[I_{\erva}^s]&=&\er [(I_{\erva})^s \, e^{-\mu \xi_{\erva}+\phi(\mu){\erva}}]\\
&=& \alpha\int_0^{\infty}
e^{-(\alpha-\phi(\mu))t}\, \er [(I_{t})^s\, e^{-\mu \xi_{t}}]\, dt\\
&=& \Gamma(s+1)/\Gamma(\phi(\cdot+\mu)-\phi(\mu)+\alpha; s+1).
\end{eqnarray*}
  Since $\alpha$  can be taken to be arbitrary, replacing $\alpha$  by $\alpha +\phi(\mu)$   in the second and last equalities and rearranging terms,
we get the second statement.
\end{proof}
\subsection{Proof of Theorem \ref{R-product} and Corollary \ref{remainder-gordon}}    \label{sect:remainders}
We have all the elements to conclude the  proof of Theorem~\ref{R-product}.
\begin{proof}[Proof of Theorem~\ref{R-product}]  An application of (\ref{PsiR}) and formula (\ref{gammaphi-0}), which is given below, gives  the first claim of the theorem.  Let us prove the identity in law of (\ref{factor-0}). Notice that we have
\[
\er \left[G^{(k)} \right]=\frac{\phi'(k)}{\phi(k)}, \quad k=1,2,
\ldots
\]
 For all $s>0$, we can write
\begin{eqnarray}
\er \left[  e^{ s  \left(- \gamma_{\phi} + \sum_{k=1}^{N}
(E[G^{(k)}]-G^{(k)}) \right) }
\right]&=&e^{-\gamma_{\phi}s}\prod_{k=1}^{N}
e^{s \frac{\phi'(k)}{\phi(k)}}\er \left[e^{-sG^{(k)}}\right]  \nonumber \\
&=&e^{-\gamma_{\phi}s}\prod_{k=1}^{N} e^{s
\frac{\phi'(k)}{\phi(k)}}\frac{\phi(k)}{\phi(s+k)}  \longrightarrow
\er\left[R^s \right]\quad \hbox{as} \quad N\rightarrow \infty,
\label{limfactor-0}
\end{eqnarray}
where, for the last equation, we used Theorem \ref{distribution-infty}
and the injectivity of Mellin transform. Let us now prove the first
equality (in law)  of (\ref{factor-0}) by using the aforementioned
result of Berg.  We have the identity
\begin{eqnarray*}
\er \left[\exp
\lambda\sum^{N}_{k=1}\left(\phi_{k}(1)-G^{(k)}\right)\right] &=&
\prod^{N}_{k=1}\er \left[\exp\lambda\left(\phi_{k}(1)-G^{(k)}\right)\right]\\
&=&\prod^{N}_{k=1}\exp-\int_{\re^{+}}\! \! \! \!H(dx)e^{-k x}\left(1-e^{-\lambda x}-\lambda (1-e^{-x})\right)\\
&=&\exp-\int_{\re^{+}} \! \! \! \! H(dx)\frac{e^{-x}\left(1-e^{-Nx}\right)}{1-e^{-x}}\left(1-e^{-\lambda x}-\lambda (1-e^{-x})\right)\\
&\longrightarrow & \exp-\int_{\re^{+}} \! \! \! \!
H(dx)\frac{e^{-x}}{1-e^{-x}}\left(1-e^{-\lambda x}-\lambda
(1-e^{-x})\right)\quad \hbox{as}\quad N\rightarrow \infty,
\end{eqnarray*}
where the last equality is obtained by using the dominated
convergence theorem. By virtue of (\ref{loggc}), which can be written in the form
\begin{equation} \label{integral-3}
\log \Gamma_\phi(s+1)=s\log{\phi(1)}+ \int_{(0,\infty)}
\left(s-\frac{1-e^{-sx}}{1-e^{-x}} \right)e^{-x}H(dx), \quad s>0,
\end{equation}
the
last term in the previous   formula  equals $ \phi (1)^{-s}\er\left
[e^{\lambda \log R} \right]$.  We conclude using the injectivity of
 the two-sided Laplace transform. Finally, let us check that
the second equality of (\ref{factor-0}) holds true. For that, we can
write
\begin{equation*}
\begin{split}
-\gamma_{\phi}+\sum_{k=1}^{N}&\left(E[G^{(k)}]-G^{(k)}\right)-\log{\phi(1)}-\sum_{k=1}^{N}\left(\phi_k(1)-G^{(k)}\right)\\
&=-\gamma_{\phi}-\log {\phi(1)}+\sum_{k=1}^{N}\left(E[G^{(k)}]-\phi_k(1)\right)\\
&=-\gamma_{\phi}-\log {\phi(1)}+\sum_{k=1}^{N} \left( \frac{\phi'(k)}{\phi(k)}+\log \er[e^{-G^{(k)}}] \right)  \\
&=-\gamma_{\phi}-\log {\phi(1)}+\sum_{k=1}^{N}\left(\frac{\phi'(k)}{\phi(k)}+\log\frac{\phi(k)}{\phi(k+1)}\right)\\
&=-\gamma_{\phi}+\sum_{k=1}^{N}\frac{\phi'(k)}{\phi(k)}-\log{\phi(N+1)}
\rightarrow 0\quad \hbox{as}\quad N\rightarrow \infty.
\end{split}
\end{equation*}
\end{proof}
\begin{proof}[Proof of Corollary~\ref{remainder-gordon}]Formula (\ref{factor-0}) can be written as
\begin{eqnarray*}
\log{R}&\stackrel{\text{(d)}}{=}&-\gamma_{\phi}+\sum_{k=1}^{n}\left(\er [G^{(k)}]-G^{(k)}\right)+\sum_{k=n+1}^{\infty}\left(\er [G^{(k)}]-G^{(k)}\right)\\
&=&-\gamma_{\phi}+\sum_{k=1}^{n}\left(\er [G^{(k)}]-G^{(k)}\right)+\sum_{k=1}^{\infty}\left(\er [G^{(k+n)}]-G^{(k+n)}\right)\\
&\stackrel{\text{(d)}}{=}&-\gamma_{\phi}+\sum_{k=1}^{n}\left(\er
[G^{(k)}]-G^{(k)}\right)+\log{R_{(n)}}+\gamma_{\phi(\cdot+n)},
\end{eqnarray*}
where $R_{(n)}$, which is assumed to be independent of the r.v.'s in
the  r.h.s. of the above formula, is the analogue of $R$ when we
work with a subordinator having the Laplace exponent $\phi(\cdot+n)$
instead of $\phi(\cdot)$. Next, we have
\begin{eqnarray*}
\gamma_{\phi(\cdot+n)}&=&\lim_{m\rightarrow \infty} \left\{\sum_{k=1}^m \frac{\phi'(k+n)}{\phi(k+n)}-\log {\phi(m+n)} \right\}\\
&=&\lim_{m\rightarrow \infty} \left\{\sum_{k=1}^{m+n}
\frac{\phi'(k)}{\phi(k)} -\log {\phi(m+n)} \right\}-\sum_{k=1}^n
\frac{\phi'(k)}{\phi(k)}\\
&=&\gamma_{\phi}-\sum_{k=1}^n \frac{\phi'(k)}{\phi(k)}.
\end{eqnarray*}
It follows that
\begin{equation}\label{Gordon-2}
\log{R}\stackrel{\text{(d)}}{=}\sum_{k=1}^{n}\left(\er
[G^{(k)}]-G^{(k)}\right)+\log\frac{R_{(n)}}{\phi(n+1)}+d_n \,,
\end{equation}
where
\begin{equation*}
d_n:=-\sum_{k=1}^n \frac{\phi'(k)}{\phi(k)}+\log \phi(n+1).
\end{equation*}
Now, we can write
\begin{eqnarray*}
d_n
&=&\log\phi(1)-\sum^{n}_{k=1}\left[\frac{\phi^{\prime}(k)}{\phi(k)}-\log\frac{\phi(k+1)}{\phi(k)}\right]\\
&=&\log\phi(1)-\sum^{n}_{k=1}\int^{\infty}_{0}e^{-(k-1)x}e^{-x}(x-(1-e^{-x}))H(dx)\\
&=&\log\phi(1)-\int_{(0,\infty)}H(dx)\frac{e^{-x}}{(1-e^{-x})}\left(e^{-x}-1+x\right)(1-e^{-nx}),
\end{eqnarray*}
where, to obtain the second equality, we used Lemma~\ref{harmonic}.
Moreover, the above calculations show that the constant
$\gamma_\phi$ defined in (\ref{EM-gen-gamma}) admits the
representation
\begin{equation}\label{gammaphi-0}
-d_{n}\xrightarrow[n\to\infty]{}\gamma_\phi =-\log \phi(1)+
\int_{(0,\infty)}\left( e^{- x}-1 +
x\right)\frac{e^{-x}}{1-e^{-x}}H(dx).
\end{equation}

The representation of the error term $B_{n}$ follows immediately
from this expression. Recall {that}  $H^{(n)}(dx)=e^{-nx}H(dx)$ is
the h.p.m.$\,$associated to $\phi(\cdot+n)$. Thus, by applying (\ref{integral-3}), for all $s>0$, we can write
  \begin{eqnarray*}
  \er\left( (R_{(n)}/\phi(n+1))^s\right)&=&
\frac{\Gamma_{\phi(\cdot+n)}(s+1)}{\phi(1+n)^s}\\
&=&\exp{\int_{\re^+} \left(s-\frac{1-e^{-sx}}{1-e^{-x}} \right)e^{-x}H^{(n)}(dx)}\\
&=&\exp{\int_{\re^+} \left(s-\frac{1-e^{-sx}}{1-e^{-x}}
\right)e^{-(n+1)x}H(dx)}\rightarrow 1 \quad \hbox{as}\quad
n\rightarrow \infty,
\end{eqnarray*}
where we applied the dominated convergence theorem. This concludes
the proof of the Corollary.\end{proof}

\section{Further properties of harmonic potentials of subordinators}\label{furtherproperties}
Before taking a look at some examples, in the next section, we make here a relatively long digression to establish some identities and properties for
harmonic potentials.  We start with the following remark.
\begin{remark}\label{digression}
 (i) We read from formula (\ref{kappa})
that the Bernstein function $\phi$ is linked with the h.p.m $H$ by the expression
\begin{equation} \label{fisurfi'}
\frac{\phi'(\lambda)}{\phi(\lambda)}=\int_{(0,\infty)}e^{-\lambda x
} \,x\,H(dx),\quad \lambda >0.
\end{equation}
 But, for every $c
>0$, $\phi$ and $c \,\phi$  share the same logarithmic derivative. Hence, they share the  same  h.p.m.\\
(ii) If $\xi$ and $N$ two independent subordinators and $\xi \circ
N$ denotes the process $\xi$ subordinated by $N$,  then the finite-dimensional
distributions of $\xi \circ N$ are recovered by
$$\pr \bigl((\xi \circ N)_t \in dx\bigr) =\int_0^\infty  \pr(\xi_s\in dx) \pr ( N_t \in ds),\quad t\geq 0.$$
Hence, the potential and harmonic potential measures for the
composition are trivially, and respectively, given by
\begin{eqnarray*}
 V_{\xi \circ N}(dx)  =\int_0^\infty  \pr(\xi_s\in dx) V_{N}(ds), \quad
 H_{\xi \circ N}(dx)  =\int_0^\infty  \pr(\xi_s\in dx)  H_{N}(ds).
\end{eqnarray*}
\end{remark}
 Let $S^\gamma,\; 0<\gamma<1,$ be a $\gamma$-stable subordinator.
That is, it has no killing term and drift, and its  L\'evy measure is given by
\begin{equation} \label{stable}cx^{-1-\gamma}\,dx,\quad x>0,
\end{equation}
where $c>0$ is a normalizing constant.
 Since the associated Bernstein function is proportional to $\lambda^\gamma$, we get, by Remark \ref{digression}~(ii),
 that the corresponding h.p.m is equal to $\gamma \,x^{-1}\, dx, \;
 x>0$. So, assume that $S^\gamma$ is independent of a subordinator $\xi$ and
define a new subordinator $\xi^{\gamma}$  by  subordinating
$\xi$  by $S^\gamma.$  Denoting by $H^{\gamma}$ the h.p.m.$\,$of $\xi^{\gamma},$ we deduce from the above discussion the following simple,  but useful, result which relates $H^{\gamma}$ and $H.$
\begin{lemma}
 It holds that
\[
H^{\gamma}(dx)=\gamma H(dx),\quad  x>0.
\]
\end{lemma}
A natural question is to know how the p.m.$\,$and the h.p.m.$\,$are related.  This question, which is answered in the following two lemmas, will lead to interesting consequences on the harmonic potential measure.
\begin{lemma}   We have the equality of measures, on $\mathbb{R}^+$,
\begin{equation*}
 xH(dx)=
\left\{
\begin{array}{ll}
  \quad  \int_{[0,x]} \Pi(dy)y V(dx-y)& \textrm{ if }\  a=0;\\
 av(x) dx+\int_{[0,x]} \Pi(dy)y v(x-y) dx&  \textrm{ if }\ a>0,
\end{array}
\right.
\end{equation*}
where $v$ denotes the density of the p.m.$\,$$V,$ which we know exists when $a>0$.
\end{lemma}
\begin{proof} We can write
\begin{equation}
\begin{split}
\int_{\re^+}e^{-\lambda x} \left(a V(dx)+\int_{[0,x]} \Pi(dy)y
V(dx-y)\right)&=\frac{a}{\phi(\lambda)}+\frac{\phi'(\lambda)-a}{\phi(\lambda)}=\frac{\phi'(\lambda)}{\phi(\lambda)}
\end{split}
\end{equation}
for all $\lambda>0$,
where we used $\phi'(\lambda)=a+\int_{\re^+} xe^{-\lambda x}
\Pi(dx)$ and formulae  (\ref{potential}) and  (\ref{fisurfi'}). The
result follows from the injectivity of the Laplace transform.
\end{proof}

\begin{lemma}\label{Potential-harmonic} (1) We have the equality of measures
\[
H(dx)=\int_{0}^{\infty}V_{\alpha}(dx)\, d\alpha ,\quad x>0.
\]
\noindent  (2) For $x>0,$ denote the first
hitting time of $x$ by $T_{\{x\}}=\inf\{t>0: \xi_{t}=x\}$. If  $\xi$
has a strictly positive drift $a$  then the h.p.m.$\,$has a density
$h$ which is given by
\begin{equation} \label{hdensity}
h(x)=\er
\left[\frac{1}{aT_{\{x\}}}{\bf 1}_{\{T_{\{x\}}<\infty\}}\right],\quad
x>0.
\end{equation}

\end{lemma}
\begin{proof} (1) This follows from the elementary identity $t^{-1}=\int_{0}^{\infty}
e^{-\alpha t} \, d\alpha$ and Fubini's Theorem.

\noindent (2) According to the discussion in \cite{Bertoin-1996},
pp.\,80-81, for  $\alpha>0$, the $\alpha$-potential of $\xi$ has
a density, say $v_{\alpha},$ which is related to the law of the
hitting time of singletons  via the formula
 \begin{equation}\label{hitting}
 \er \left[\exp\{-\alpha T_{\{x\}}\}{\bf 1}_{\{T_{\{x\}}<\infty\}}\right]=av_{\alpha}(x),\quad x>0.
 \end{equation}
 The result follows from a combination of  assertion (1) and formula (\ref{hitting}).
\end{proof}

We emphasize that the expression of the h.p.m.$\,$given in  formula (\ref{hdensity}) above, for the case when $a>0$, is reminiscent of Kingman's well known result ensuring that the p.m.$\,$has a bounded density which is
  is proportional to the creeping probability.

By an application of Lemma~\ref{Potential-harmonic} together with
condition {\it (v)} of Theorem~\ref{exp-undershoot-id}, we obtain
 Corollary~\ref{cor:1}.
\begin{proof}[Proof of Corollary~\ref{cor:1}]\label{corollary002}
Observe that if the drift $a$ of $\xi$ is strictly positive then necessarily
\[
T_{x}\leq x/a \quad  \text{on}\quad \{T_{\{x\}}<\infty\}.
\]
Hence, the h.p.m.$\,$has a density, say $h,$ such that $xh(x)\geq 1$ for all $x>0$. This has as a consequence that the random variable $\log{I}$ is not i.d. This is true because, according to  the assertion  (v) of Theorem~\ref{exp-undershoot-id}, a necessary and sufficient condition for this to
hold is that $\kappa(dx)\leq dx.$ But, in this setting, we have $\kappa(dx)=xh(x)dx\geq dx$ for all $x>0$.
\end{proof}

A combination of Lemma \ref{Potential-harmonic} and the results of Gripenberg~\cite{gripenberg1,gripenberg2}, Friedman~\cite{friedman}, Hawkes~\cite{hawkes} and Hirsch~\cite{hirsch}  allows us to establish, without much effort, the following proposition.
\begin{prop}\label{log-convexity}
If  the L\'evy tail $\overline{\Pi}$  is log-convex  then the h.p.m.$\,$$H$ has a density on $(0,\infty)$ which is non-increasing; if furthermore the former function is c.m.$\,$then the density of $H$ on $(0,\infty)$ is also c.m.\\
\end{prop}
\begin{remark}In fact, more can be said in the above case. It can be verified that $x\mapsto \overline{\Pi}(x)$ is c.m.$\,$if and only if $\Pi$ has a completely monotone density, which is equivalent to require that $\phi$ is a complete Bernstein function. This,  in turn, is a necessary and sufficient condition for the density of $H$, on $(0,\infty)$, to be the Laplace transform of some  function $\eta$, on $(0,\infty),$ taking values in $[0,1].$ For further details about these statements, see the monograph of Schilling et al.$\,$(\cite{SSV}, Theorem 6.10, p. 58).
\end{remark}
\begin{proof}[Proof of Proposition~\ref{log-convexity}]
Let $\alpha>0$ be fixed. We will prove that under the assumptions of the proposition, the $\alpha$-potential of $\xi$, $V_{\alpha}$, has a non-increasing (c.m.) density. The result will follow from Lemma~\ref{Potential-harmonic}. Let us recall, from the discussion in Subsection~\ref{subsectionundershoot}, that the $\alpha$-potential of the subordinator $\xi$, whose characteristic triplet is $(q, a, \Pi),$ corresponds to the $0$-potential of a subordinator with characteristic triplet $(q+\alpha, a, \Pi)$. With this remark at hand, the arguments in the proof of Lemma~\ref{exp-undershoot} allow us to ensure  that, by integrating  (\ref{volterraeq}) over $[0,t]$, we obtain
\[
1=av_{\alpha}(t)+\int^{t}_{0}\left(\Pi(t-s, \infty)+q+\alpha\right)V_{\alpha}(ds),\quad t>0.
\]
According to  \cite{gripenberg1,gripenberg2,friedman,hawkes,hirsch}, it follows that $V_{\alpha}$ has a non-increasing (c.m.) density, on $(0,\infty)$, whenever $q+\alpha+\Pi(x, \infty)$ is a log-convex (c.m.) function. But this is a straightforward consequence of the assumption that $\overline\Pi(\cdot)$ bears this property.
\end{proof}

We prove  Corollary~\ref{cor:2} by applying the first result of  Theorem~\ref{R-product} and
Proposition~\ref{log-convexity}.

\begin{proof}[Proof of Corollary~\ref{cor:2}]
According to the latter  lemma, we just need to notice
that $x\mapsto e^{-x}/(1-e^{-x})$, $x>0$, is a c.m.$\,$function.
This is true  because, on the one hand, the product of two
c.m.$\,$functions is a c.m.$\,$function. On the other hand, $e^{-x}$
is the Laplace transform of the  Dirac's delta measure concentrated
at $1$ while $(1-e^{-x})^{-1}$ is the Laplace transform of the
p.m.$\,$of a Poisson process with rate one. We conclude by  applying
Theorem 15.10 in \cite{sato} and Theorem 3.1.1 in \cite{bondesson},
respectively.
\end{proof}

\section{Examples}\label{section:examples}In this section, we will apply our results for some particular subordinators. In the first example, we obtain some general results for {\it special subordinators}. This is a very rich class of subordinators, which has been deeply studied by several authors, we refer to~\cite{SSV} for background. Some properties of exponential functionals of special subordinators have been obtained in \cite{Hirsch-Yor}.  In fact, some of the results we obtain have also been established in that paper with some relatively different arguments. Some of the specific families of subordinators,  which are considered  in the following examples, are actually special.

\begin{example}\label{example-6}  We recall that $\phi$ is said to be  a special Bernstein function if the function $\phi^{*}$ defined by
\[
\phi^{*}(\lambda)=\frac{\lambda}{\phi(\lambda)},\quad \lambda>0,
 \]
 is also a Bernstein function. In this case, we say that $(\phi,\phi^{*})$ is a conjugate pair of Bernstein functions and we denote by $(q^{*}, a^{*}, \Pi^{*})$ the characteristic triplet of $\phi^{*},$ that is
\[
\phi^{*}(\lambda)=q^{*}+\lambda a^{*}+\int_{(0,\infty)}(1-e^{-\lambda x})\Pi^{*}(dx),\quad \lambda\geq 0.
\]
 It is well known that a necessary and sufficient condition for $\phi$ to be special is for the p.m.$\,$of $\phi$ to be such that $V(dx){\bf 1}_{\{x\in(0,\infty)\}}=v(x){\bf 1}_{\{x\in(0,\infty)\}}dx,$ where $v:(0,\infty)\to\re^{+}$ is a non-increasing function. Sufficient conditions on the L\'evy measure are also known, see e.g. Hirsch \cite{hirsch} and Song and Vondracek \cite{Song-Vondracek-2006}. It is known, and easy to see, that for a special subordinator (resp.$\,$its conjugate) with Laplace exponent $\phi$ (resp.$\,$$\phi^{*}$), we have
\[
qq^{*}=0,\quad aa^{*}=0,\quad {\Pi}^{*}(x,\infty)+q^{*}=v(x), \qquad x>0,
\]
because
\[
\frac{\phi^{*}(\lambda)}{\lambda}=\frac{1}{\phi(\lambda)}=a^{*}+\int_{\re^+}e^{-\lambda x}v(x)\, dx,\quad \lambda\geq 0.
\]
By Lemma~\ref{exp-undershoot}, we already know that $G_{\erva}$ and $\erva- G_{\erva}$ are independent and that
\[
\er [e^{-\lambda
\erva}]=\frac{\alpha}{\alpha+\lambda}=\frac{\phi(\alpha)}{\phi(\alpha+\lambda)}\frac{\phi^{*}(\alpha)}{\phi^{*}(\alpha+\lambda)},\quad
\alpha>0, \; \lambda\geq 0.
\]
  Let $(\xi, \xi^*)$  be a a conjugated pair of subordinators associated to the pair $(\phi, \phi^*)$. For convenience, all objects associated to $\xi^*$, which we defined for $\xi$, are denoted with the superscript $*$.
  We immediately see the known fact that
  \[
  \erva- G_{\erva}\stackrel{\text{(d)}}{=}G^{*}_{\erva}\quad \hbox{and} \quad G_{\erva}+G^{*}_{\erva}\stackrel{\text{(d)}}{=}\erva,
  \]
where $G_{\erva}$ and $G^{*}_{\erva}$ are independent. Furthermore,  there exists a function $\rho:(0,\infty)\to[0,1]$ such that
\[
H(dx)=\rho(x)\frac{dx}{x}\quad \hbox{and} \quad H^*(dx)=(1-\rho(x))\frac{dx}{x},\quad x>0.
\]
 Moreover, $\rho$ is related to $\phi,$ $\Pi$ and $v$ as follows
\[
\int_{(0,\infty)}dy\, \rho(y)e^{-\lambda y}=\frac{\phi'(\lambda)}{\phi(\lambda)},\quad \lambda > 0,
\]
and
\[
\rho(x)=av(x)+\int_{[0,x]}\Pi(dy)yv(x-y),\quad x>0.
\]
It is worth pointing out that, in this case, we have the identities $I\stackrel{\text{(d)}}{=}R^{*}$ and $I^{*}\stackrel{\text{(d)}}{=}R,$ with obvious notations. Therefore, all the results obtained for the remainder random variable $R$ can be applied to obtain information about the exponential functional $I$. For instance, this can be used for the r.v.$\,$$\log I$ to determine when it has a self-decomposable or extended generalized gamma convolution distribution, to write the analogue of (\ref{factor-0})  of Theorem \ref{R-product} for it and so on. To prove the above claimed facts about $H$,
 recall that, since $(\phi,\phi^*)$ is a conjugated pair of Bernstein functions, equation (\ref{eq:1}), applied to $\phi$ and $\phi^*,$ gives that, for all  $\alpha>0$ and $\lambda\geq 0,$ we have
\begin{eqnarray*}
\exp\left\{-\int_{(0,\infty)}\frac{dx}{x}e^{-\alpha x}(1-e^{-\lambda
x})\right\} &=&\frac{\alpha}{\alpha+\lambda}=  \frac{\phi(\alpha)}{\phi(\alpha+\lambda)}\frac{\phi^{*}(\alpha)}{\phi^{*}(\alpha+\lambda)}\\
&=&\exp\left\{-\int_{(0,\infty)}H(dx)e^{-\alpha x}(1-e^{-\lambda
x})\right\} \times \\&& \quad
\exp\left\{-\int_{(0,\infty)}H^{*}(dx)e^{-\alpha x}(1-e^{-\lambda
x})\right\}.
\end{eqnarray*}
We deduce that, for each $\alpha>0,$ the L\'evy measures
$e^{-\alpha x}H(dx)+e^{-\alpha x}H^{*}(dx)$ and $e^{-\alpha x}x^{-1}dx$, on $\mathbb{R}^+$, are equal.  It follows that
\begin{equation}\label{h.p.m.}
H(dx)+H^{*}(dx)=x^{-1}dx, \quad x>0,
 \end{equation}
 and,  $H(dx)$ and $H^{*}(dx)$ are both absolutely continuous with respect to $x^{-1}dx$.  Thus, there exists a density function $\rho$ such that $\rho(x)x^{-1}dx=H(dx)$ with $0\leq \rho(x)\leq 1$, $x>0$. Analogously, $H^*$ can be written as  $H^*(dx)=\rho^{*}(x)x^{-1}dx$ for some function $\rho^*$ satisfying  $0\leq \rho^{*}(x)\leq 1$, $x>0$. Furthermore,  it follows from (\ref{h.p.m.}) that $\rho(x)+\rho^{*}(x)=1$ for a.e.  $x>0$.  The rest follows from the results in Section~\ref{furtherproperties}.

The processes of Examples \ref{example-00}, \ref{example-1} and \ref{example-5}
are particular cases of special subordinators.
\end{example}

\begin{example}\label{example-00}
Complete Bernstein functions, which are necessarily special, are those
Bernstein functions  for which the corresponding L\'evy measures  have  completely monotone densities,  see e.g.
\cite{SSV}, Chapter 6. These have the following representation
\begin{equation}\label{cbf}
 \frac{\phi(\lambda )}{\phi(1)}= \exp    \int_0^\infty   \left( \frac{1}{1+ t} -\frac{1}{\lambda + t} \right)  \eta(t)\, dt
= \exp    \int_0^\infty     \frac{\lambda - 1}{\lambda + t} \,
\frac{\eta(t)}{1+t}\, dt
\end{equation}
 for all $ \lambda \geq 0$, were $\eta :\mathbb{R}_+\rightarrow [0,1]$ is a measurable function. In this case, the density of the h.p.m.$\,$is the Laplace transform of $\eta$ i.e.
 $\rho(x)=\int_0^{\infty} e^{-xt}\eta(t)\, dt$.  Marchal \cite{Marchal} considered this class of Bernstein functions using the representation
\[
\phi^{(\alpha)}(\lambda ) =   \exp  \int_0^1    \frac{\lambda -
1}{1+(\lambda-1) x} \,\alpha(x)\, dx.
\]
Making the substitution $x=1/(1+t)$, we obtain that
\[
\phi^{(\alpha)}(\lambda ) =  \exp  \int_0^1    \frac{\lambda - 1}{
\lambda +t} \,\frac{\alpha\left(1/(1+t)\right)}{1+t}\, dt,
\]
which is, indeed, of the form (\ref{cbf}) with
$\eta(t)=\alpha\left(1/(1+t)\right) {\bf 1}_{[0,1]}(t)$.
Furthermore, in Section 16.11 of \cite{SSV}, we can find several
interesting families of complete subordinators  whose harmonic potential
densities are known explicitly.
\end{example}

\begin{example}\label{example-1} For a fixed $K>0$, let $\xi$ be the deterministic  subordinator  $\xi_t = K t,\; t\geq 0$, killed at rate $q\geq 0.$ Then, the corresponding Bernstein function is   $\phi (\lambda)=q+K\lambda$. The associated p.m.$\,$and h.p.m.$\,$are, respectively, given by
\[
V(dx)=K^{-1} e^{-qx/K} dx, \quad \mbox{and}  \quad H(dx)=e^{-qx/K}
\frac{dx}{x},\quad x>0.
\]
  In this case, $I\stackrel{\text{(d)}}{=}K^{-1}{\bf B}_{1,q/K}$  where $\,$${\bf B}_{1,q/K}$ is a r.v.$\,$following the Beta$(1,q/K)$-law.  This subordinator is special  and the Bernstein function corresponding to its conjugate is
  \[
\phi^{*}(\lambda):=\frac{\lambda}{\lambda K+q}=\frac{1}{K}\int_0^{\infty} (1-e^{-\lambda x})\left( (q/K)e^{-(q/K)x}\right) dx.
\]
 That  is $\xi^*$   is a compound Poisson process with jumps of exponential size of parameter $q/K$ and arrival rate $1/K.$  Its h.p.m.$\,$is hence given by
\[
\quad H^{*}(dx)=(1-e^{-qx/K})\,\frac{dx}{x},\quad x>0.
 \]
 To see this, note that
 \[
 \frac{{\phi^*}'(\lambda)}{\phi^*(\lambda)}=\frac{q}{\lambda} \frac{1}{\lambda K+q}=\int_0^{\infty} e^{-\lambda x} (1-e^{-qx/K})dx, \qquad \lambda\geq 0,
 \]
 and conclude using Lemma \ref{harmonic} and the injectivity of the Laplace transform.
 From standard properties of the beta and gamma distributions and some factorizations of the exponential distribution, see e.g.$\,$\cite{BY2001}, it is easy to see that the remainder random variable $R$ has the same law as $K\, \boldsymbol{\gamma}_{1+q/K}$   where this gamma variable has the p.d.f. ${\Gamma(1+ q/K)}^{-1}  x^{q/K}e^{-x},\; x>0$. This assertion can also be easily verified by calculating the moments of $R.$ Taking $K=1$, in this setting, we recover, from identity (\ref{factor-0}),  Gordon's representation of a log-gamma r.v.$\,$with shape parameter $t=1+q.$ Indeed, let $(\mathbf{e}^{(i)}, i\geq 1)$ be an i.i.d.$\,$sequence of r.v.'s following a standard exponential distribution. For any $k\geq 1$, the random variable $G_{\mathbf{e}_{k}}$ follows the exponential distribution with parameter $q+k,$ therefore $$\mathbf{E}(G_{\mathbf{e}_{k}})-G_{\mathbf{e}_{k}}\stackrel{\text{(d)}}{=}\frac{1}{q+k}-\frac{\mathbf{e}^{(k)}}{q+k}.$$
 Identity (\ref{factor-0}) becomes
\begin{equation}\label{Gordon-representation}
\begin{split}
\log \boldsymbol{\gamma}_{1+q}\stackrel{\text{(d)}}{=}&\log R\stackrel{\text{(d)}}{=}-\gamma_{\phi}+\sum^{\infty}_{j=0}\frac{1}{(q+1)+j}-\frac{\mathbf{e}^{(k)}}{(q+1)+j}\\
&=-\left(\gamma_{\phi}+\sum^{\infty}_{k=1}\frac{q}{k(q+k)}\right)+\sum^{\infty}_{j=0}\frac{1}{1+j}-\frac{\mathbf{e}^{(k)}}{(q+1)+j}
\end{split}
\end{equation}
 and it is readily verified that the constant $\gamma_{\phi}+\sum^{\infty}_{k=1}\frac{q}{k(q+k)}$ equals Euler's number. Formula  (\ref{Gordon-representation}) is obtained by Gordon \cite{gordon} in his Theorem 2, formula (3).

Finally, for $0<\gamma<1$, the measures
\[
H_{\gamma} (dx)=\gamma\frac{ e^{-qx}}{x}\,dx\quad \text{and}\quad
H^{*}_{\gamma}(dx)=\gamma (1-e^{-qx})\,\frac{dx}{x},\quad x>0,
 \]
 are the h.p.m.'s of subordinators whose Laplace exponents are  given by the Complete Bernstein functions $\lambda \mapsto(\lambda + q)^\gamma$ and $\lambda^{\gamma}/(\lambda + q)^\gamma$, respectively.
 These  naturally arise  when subordinating, with a $\gamma$-stable subordinator, i.e.$\,$the subordinator described just before (\ref{stable}).
\end{example}

\begin{example}   Recall that the so-called Mittag-Leffler function $E_\alpha,\; \alpha >0$, is defined
by
\begin{equation}\label{mittag}
E_\alpha(z)= \sum_{k \geq 0} \frac{z^k}{\Gamma (k \alpha +1)}, \quad
z\in \mathbb{C}.
\end{equation}
This corresponds, in the case $0<\alpha <1$, to
$$E_\alpha(z)=\er\Big[e^{\,z\,(S^\alpha_1)^{-\alpha}}\Big]\,,$$
where $S^\alpha$ is an $\alpha$-stable subordinator.
\begin{prop}\label{stagamma}
The h.p.m. $\,H_{q}^{\alpha}(dx)\,$ associated
to the subordinator ${(X_s)}_{ s\geq 0},$ obtained by killing a
standard $\alpha$-stable subordinator $S^\alpha$, at rate $q\geq 0$,  is
\[
H_{q}^{\alpha}(dx)=\alpha  \, E_\alpha(-q
x^\alpha)  \frac{dx}{x},\quad x>0.
\]
\end{prop}

\begin{proof}
Let $f:\re^{+}\to \re^{+}$ be a   measurable function. Using the scaling property, Fubini's theorem and the  change of variables $s=t^{1/\alpha}S_{1},$ we obtain the following formula for the h.p.m.$\,$of $X$
\begin{equation}
\begin{split}
\int_{[0,\infty)}H_{q}^{\alpha}(dx)f(x)&:=\int^{\infty}_{0}\frac{dt}{t}\er \left[f(X_{t})\right]\\
&= \int^{\infty}_{0}\frac{dt}{t}e^{-qt}\er \left[f(t^{1/\alpha}S^\alpha_{1})\right]\\
&=\int^{\infty}_{0}\frac{ds}{s}\alpha\er \left [\exp\{-q (s
/S^\alpha_{1})^{\alpha} \}\right]f(s)
\end{split}
\end{equation}
and the result follows.\end{proof}

As a consequence, we get that $\log{I}$, with $I:=\int^{\infty}_{0}e^{-X_{s}}ds$, belongs to the $B(\re)$-class, see
Barndorff-Nielsen et al.$\,$\cite{BNMS}, and its L\'evy measure has the density
\[
\frac{e^{x}}{|x|(1-e^{x})}\left((1-\alpha)+\alpha \er
\left[1-e^{-q(-x)^{\alpha} (S^\alpha_{1})^{-\alpha}} \right]\right),
\quad x<0.
 \]
 Similarly, the random variable $\log{R}$, associated to $X,$ belongs to the class of extended generalized gamma convolutions.

 As a consequence of the above remark,  the L\'evy measure of the distribution of $\log{R}$  is absolutely continuous with respect to the probability distribution of a one sided Linnik r.v.$\,$We think that this link  deserves to be studied in further detail.
\end{example}

\begin{example}\label{eaxample:6}
The following Lemma is useful for the construction of explicit examples of  L\'evy measures of  remainder random variables.  Its proof is easy and is based on elementary properties of Poisson processes, so we omit the details.
\begin{lemma}\label{EHP}
Assume $\xi$ is a subordinator with drift $a\geq 0,$ no killing term
and finite L\'evy measure $\Pi.$ Let $(N_{t}, t\geq 0)$ be a Poisson
process with intensity $\Pi(0,\infty)=:c,$ $(Y_{i})_{i\geq 1}$
i.i.d.$\,$random variables which are independent of $N,$ with common
distribution $c^{-1}\Pi({\rm d}x),$ and $S_{n}=\sum^{n}_{i=1}Y_{i},$
$n\geq 0.$ Hence $\xi_{t}=a\,t+S_{N_{t}}=a\,
t+\sum^{N_{t}}_{i=1}Y_{i},$ for $t\geq 0.$ When the drift $\,a$ is
strictly positive, we have the equality of measures
$$\int_{(0,\infty)}\frac{{\rm d}t}{t}\pr(\xi_{t}\in {\rm d}x)=
\frac{c}{a} \sum^{\infty}_{n=0}\frac{1}{n!}\er
\left[\left(\frac{c}{a}(x-S_{n})\right)^{n-1}\exp\left\{-\frac{c}{a}(x-S_{n})\right\}{\bf
1}_{\{S_{n}\leq x\}}\right]\,{\rm d}x$$ on $(0,\infty),$ with the
convention $S_0\equiv0$. While, when  the drift $a$ is zero,  we
have
\[
\int^{\infty}_{0}\frac{{\rm d}t}{t}\pr (\xi_{t}\in {\rm d}x)=\sum_{n\geq 1}\frac{1}{n}\pr (S_{n}\in dx)
\]
on $\ (0,\infty)$; that is, the h.p.m.$\,$of $\xi$ coincides with the harmonic renewal measure of the random walk $(S_{n}, n\geq 1).$
\end{lemma}

Assume for instance that $\xi$ is a subordinator with no killing, drift $1,$ and L\'evy measure
\[
\Pi({\rm d}x)=\frac{c}{\Gamma(\beta)} x^{\beta-1}e^{-x}{\rm
d}x,\quad x>0,
 \]
 where $c,\,\beta >0$. Equivalently, the Laplace exponent is given by $\lambda\mapsto \lambda+c (1- \frac{1}{(1+\lambda)^\beta})$.  It follows from Lemma \ref{EHP} and elementary calculations that the h.p.m.$\,$of $\xi$ is given by
\[
\left(1+\sum^{\infty}_{n=1}\frac{c^{n}
x^{n(1+\beta)}\Gamma(n\beta)}{n\Gamma(n(1+\beta))}\right)e^{-c
x}x^{-1}dx,\quad x>0.
 \]
 Furthermore, assuming the same jump structure and no drift, i.e.$\,$the Laplace exponent is  $\lambda\mapsto c (1- \frac{1}{(1+\lambda)^\beta}),$ we find the h.p.m.
 \[
 \beta\left(E_{\beta}(x^{\beta})-1 \right) e^{-x}x^{-1}dx ,\quad x>0. \]

\end{example}

\begin{example}\label{example-5}
 Assume that $\xi$ is the ascending ladder time, with Laplace exponent $\phi$, of a general real valued L\'evy process  $X$. It follows from the Lemma~\ref{exp-undershoot} and the identity (5) in page 166 of \cite{Bertoin-1996}, which is a consequence of Fristedt's formula, that the h.p.m.$\,$of $\xi$ is given by
\begin{equation}\label{hpm-stable}
H(dx)= \pr(X_x\geq 0)\, \frac{dx}{x},\quad x>0.
 \end{equation}
 We further note that, in the stable case, we have  $\pr (X_t\geq 0)=\beta$ for all $t>0,$ for some  $0<\beta<1$ which is called the positivity index of $X.$ Recall that $\er [\erv^{s}]=\Gamma(s+1)$, where $\erv$ is a standard exponential random variable, and
\begin{equation}\label{integral-gamma}
\log \Gamma(s+1)=\int_0^{\infty}
\left(s-\frac{1-e^{-st}}{1-e^{-t}} \right)e^{-t}\frac{dt}{t}
\end{equation}
which is found for example in Theorem 1.6.2 in page 28 of
\cite{aar}. It follows from a combination of (\ref{loggc}),
(\ref{hpm-stable}) and (\ref{integral-gamma})  that
$\Gamma_\phi(s)=(\Gamma(s))^{\beta}$ and
$\Gamma^*(\phi(\cdot);s)=\Gamma_{\phi^{*}}(s)=(\Gamma(s))^{1-\beta}$.
The characteristics of $\log R$ and $\log I$ are easy to obtain from
the above expressions.
\end{example}

\noindent{\bf Acknowledgment:} We are grateful to
anonymous referees for their valuable comments that improved the
presentation of this paper. We are also indebted to l'Agence
Nationale de la Recherche for  the research grant
ANR-09-Blan-0084-01. The work of the second author is supported by
King Saud University, Riyadh, DSFP, through the grant MATH 01. The
third author acknowledges support from CONACYT grant no.~134195
Project {\it Teor{\'\i }a y aplicaciones de procesos de L\'evy}, and
thanks the support provided by the CEREMADE CNRS-UMR 7534 where he
concluded this work.

\end{document}